\documentclass[11pt]{article}

\usepackage{amsmath}
\usepackage{amsfonts}
\usepackage{amssymb}
\usepackage{amsthm,color}
\usepackage{newlfont}
\usepackage{bbm}
\usepackage{graphicx}

\newtheorem{Theorem}{Theorem}[section]

\newtheorem{Lemma}[Theorem]{Lemma}
\newtheorem{Corollary}[Theorem]{Corollary}
\newtheorem{Remark}[Theorem]{Remark}

\numberwithin{equation}{section}

\begin{document}

\def\le{\left}
\def\r{\right}
\def\cost{\mbox{const}}
\def\a{\alpha}
\def\d{\delta}
\def\ph{\varphi}
\def\e{\varepsilon}
\def\la{\lambda}
\def\si{\sigma}
\def\La{\Lambda}
\def\B{{\cal B}}
\def\A{{\mathcal A}}
\def\L{{\mathcal L}}
\def\O{{\mathcal O}}
\def\bO{\overline{{\mathcal O}}}
\def\F{{\mathcal F}}
\def\K{{\mathcal K}}
\def\H{{\mathcal H}}
\def\D{{\mathcal D}}
\def\C{{\mathcal C}}
\def\M{{\mathcal M}}
\def\N{{\mathcal N}}
\def\G{{\mathcal G}}
\def\T{{\mathcal T}}
\def\R{{\mathbb R}}
\def\I{{\mathcal I}}

\def\bw{\overline{W}}
\def\phin{\|\varphi\|_{0}}
\def\s0t{\sup_{t \in [0,T]}}
\def\lt{\lim_{t\rightarrow 0}}
\def\iot{\int_{0}^{t}}
\def\ioi{\int_0^{+\infty}}
\def\ds{\displaystyle}
\def\pag{\vfill\eject}
\def\fine{\par\vfill\supereject\end}
\def\acapo{\hfill\break}

\def\beq{\begin{equation}}
\def\eeq{\end{equation}}
\def\barr{\begin{array}}
\def\earr{\end{array}}
\def\vs{\vspace{.1mm}   \\}
\def\rd{\reals\,^{d}}
\def\rn{\reals\,^{n}}
\def\rr{\reals\,^{r}}
\def\bD{\overline{{\mathcal D}}}
\newcommand{\dimo}{\hfill \break {\bf Proof - }}
\newcommand{\nat}{\mathbb N}
\newcommand{\E}{\mathrm E}
\newcommand{\Pro}{\mathbb P}
\newcommand{\com}{{\scriptstyle \circ}}
\newcommand{\reals}{\mathbb R}

\def\Amu{{A_\mu}}
\def\Qmu{{Q_\mu}}
\def\Smu{{S_\mu}}
\def\H{{\mathcal{H}}}
\def\Im{{\textnormal{Im }}}
\def\Tr{{\textnormal{Tr}}}
\def\E{{\mathrm{E}}}
\def\P{{\mathbb{P}}}

\title{On dynamical systems perturbed by a null-recurrent fast motion: The continuous coefficient case with independent driving noises}

\author{Zsolt Pajor-Gyulai, Michael Salins\\
\vspace{.1cm}\\
Department of Mathematics\\
 University of Maryland\\
College Park\\
 Maryland, USA
}

\date{}

\maketitle

\begin{abstract}
An ordinary differential equation perturbed by a null-recurrent diffusion will be considered in the case where the averaging type perturbation is strong only when a fast motion is close to the origin. The normal deviations of these solutions from the averaged motion are studied, and a central limit type theorem is proved. The limit process satisfies a linear equation driven by a Brownian motion time changed by the local time of the fast motion.
\end{abstract}

\section{Introduction}
Many mathematical models arising from physics, biology, finance or other areas of science involve subsystems evolving on different time scales. Often, there is a fast and a slow component and the limiting behavior of the latter is an interesting non-trivial problem.

One possible setting is a system of diffusion processes $(X^{\varepsilon}(t),Y^{\varepsilon}(t))\in \mathbb{R}^{1+d}$ satisfying the stochastic differential equation
\begin{align*}
dX^{\e}(t)&=\frac{1}{\e^2}\phi(X^{\e}(t),Y^{\e}(t))dt+\frac{1}{\varepsilon}\varphi(X^{\e}(t),Y^{\e}(t))dW(t),&X(0)=x_0,\\
dY^{\e}(t)&=b(X^{\e}(t),Y^{\e}(t))dt+\sigma(X^{\e}(t),Y^{\e}(t))dW(t),&Y(0)=y_0,
\end{align*}
where $b$ is a $d$-dimensional vector function and $\phi$ is a one-dimensional vector function, $W(t)$ is an $r$-dimensional Brownian motion, and $\varphi$, $\sigma$ are $1\times r$ and $d\times r$ matrix valued functions respectively. This process depends on a parameter $\varepsilon$ representing the ratio of the two time scales. In other words, $X^{\varepsilon}$ changes faster and faster in time as $\varepsilon\to 0$, while $Y^{\e}$ changes on the same time scale for all values of $\e$. $X^{\e}$ and $Y^{\e}$ are referred to as the fast and the slow component respectively.

The case when the fast motion has a finite invariant measure $\mu$ was first studied by Khasminskii (\cite{K68}) and is well understood by now. He proved that, as $\varepsilon\to 0$, the law of the slow component $Y^{\e}$ approaches that of a limiting diffusion $\bar{Y}$, and one can obtain the effective drift and diffusion coefficients of $\bar{Y}$ by averaging $b$ and $\sigma$ in the first variable with respect to $\mu$. This result was later extended and refined by a vast number of authors (see e.g \cite{FW12},\cite{GS82},\cite{K04},\cite{KY04},\cite{P77},\cite{S89}, or the monograph \cite{PS08}).

Much less is known about what happens when the fast motion does not posess a finite invariant measure. In the case when the process is null recurrent, and there is a $\sigma$-finite invariant measure, naive intuition would suggest that the limit (if it exists) would be a diffusion with coefficients averaged with respect to this measure. However, this is false due to the fact that $X^{\e}$ spends most of the time in the neighborhood of infinity. It was shown in \cite{KK04} that if $\phi=0$ and there exist constants $c_1,c_2\in (0,\infty)$ such that
\[
c_1\leq\sum_{i=1}^r\varphi_i^2(x,y)\leq c_2\qquad \forall (x,y)\in\mathbb{R}^{1+d},
\]
then the asymptotic behavior of $Y^{\e}$ is governed by values of $b$ and $\sigma$ when $|x|$ is very large (under the assumption that they are non-zero). As, in general, these values can be different for positive and negative values of $x$, $Y^{\e}$ does not converge to a Markov process. Indeed, it was shown that, under certain assumptions on the large $|x|$ behavior of $\varphi$, $b$, and $\sigma$, the pair $(X^{\e}(t),Y^{\e}(t))$ converges weakly to a $(1+d)$-dimensional diffusion with diffusion coefficient discontinous at $x=0$.

In this paper, we are interested in the case when $\phi=0$, $b(x,y)=b_1(y)+b_2(x,y)$ where $b_2$ and $\sigma$ are very small as $|x|\to\infty$. The result cited above (\cite{KK04}) then implies that $Y^{\e}$ converges to the solution of the ordinary differential equation $\dot{y}=b_1(y)$ (see also Lemma \ref{lem:Y_close_to_y}), and the slow motion can be considered as a perturbation of this ODE. This situation can be intuitively understood by noting that a typical trajectory of the process $X^{\epsilon}$ is of order $1/\varepsilon$, and therefore $b_2$ and $\sigma$ are small. This implies that $Y^{\varepsilon}(t)$ cannot deviate significantly from the unperturbed solution on a finite time scale. We show that the first correction term is $\mathcal{O}(\e^{1/2})$ and study the limiting behavior of the process $\varepsilon^{-1/2}(Y^{\e}(t)-y(t))$ where $y(t)$ is the solution of the unperturbed system. The main ingredient is that the bulk of the deviation comes from the displacement when $X^{\e}$ is at distance $\mathcal{O}(1)$ from the origin, which suggests that the natural time scale to look at is defined by the local time $L^{X}(t,x)$ at $x=0$. This implies that in order to obtain a Markov process in the limit, it is necessary to keep track of both component. Indeed, we derive a limit theorem for the pair $(\e X^{\e},Y^{\e})$.

In the general case, the large $|x|$ behavior of $\varphi$ can be different depending on the sign of $x$. This creates a discontinuity  in the diffusion coefficient of the limit of $\e X^{\e}$ and one has to consider convergence to a limiting process with a certain boundary behaviour at $x=0$. Another difficulty is posed by the fact that the noise driving the fast and the slow motion are the same. In the absence of these complications, however, a proof using only elementary stochastic calculus is possible. Our result captures the phenomenon without much technical difficulties. Therefore, we consider a simplified system in this paper and return to the general case in an upcoming publication.

One motivation to study these systems is to describe certain systems with an interface where significant transport is possible only in a thin layer, see e.g. \cite{HM11} for a recent result.

This paper is organized as follows. After we state our precise result in Section \ref{sec:main_result}, we prove some preliminary lemmas in Section \ref{sec:aux_lemma}. In Section \ref{sec:reducing_to_simple}, we show that the case $\varphi\equiv 1$ can be reduced to a simpler problem which is solved in Section \ref{sec:proof_of_main_result}. We extend the result to non-constant $\varphi$ in Section \ref{sec:non_unit_varphi}.

\section{The main result}\label{sec:main_result}

We begin by stating the result in the special case where the fast motion is a Brownian motion. Let $W_1$ and $W_2$ be independent one and $d$-dimensional Brownian motions and consider the following $d$-dimensional non-homogeneous stochastic differential equation
\begin{align}
\label{eq:Y_eq} dY^{\varepsilon}(t)=[b_1(Y^{\varepsilon}(t))+b_2(\epsilon^{-1}W_1(t),Y^{\varepsilon}(t))]dt+\sigma(\epsilon^{-1}W_1(t),Y^{\varepsilon}(t))dW_2(t),
\end{align}
with initial condition $Y^{\varepsilon}(0)=y_0$.

Assume that
\[
\hat{b}(x):=\sup_{y\in\mathbb{R^d}}|b_2(x,y)|_{\reals^d}\in L^1(\mathbb{R}),\qquad \hat{\sigma}^2(x)=\sup_{y\in\mathbb{R^d}}\Tr\sigma\sigma^T(x,y)=\sup_{y\in\mathbb{R^d}}\sum_{i,j=1}^d\sigma_{ij}^2(x,y) \in L^1(\mathbb{R}).
\]
We also assume that $b_2(x,y)$ and $\sigma(x,y)$ are globally Lipschitz continuous in $x$ and $y$, and that $b_1(y)$ is twice continuously differentiable with bounded derivatives. It follows from the above that the ordinary differential equation
\[
\frac{dy}{dt}=b_1(y(t)),\qquad y(0)=y_0,
\]
which serves as the unperturbed part of the slow motion, has a  unique solution defined for all times.

To describe the limiting process, let us introduce the process $V(t)=\bar{W}_2(L^{\bar{W}_1}(t,0))$ where $\bar{W}_1$ and $\bar{W}_2$ are independent $1$ and $d$ dimensional Brownian motions respectively and $L^{\bar{W}_1}(t,0)$ is the local time of $\bar{W}_1$ at $0$. It is a continuous, non-Markovian process that only grows on a set of Lebesgue measure zero with probability one.  Note that the non-Markovity is rather innocent in this case as the pair $(\bar{W}_1,V)$ is Markovian. Also note that the conditional law $V|\bar{W}_1$ is Gaussian.

\begin{Remark}
We mention that $V$ is a known process, in the literature it is called $1/2$-fractional kinetic process and it appears as the scaling limit of certain randomly trapped random walks (see \cite{BAC07}). The connection is intuitively explained by considering the time the fast process spends away from the origin as a trapping for the slow component with a heavy tail trapping time (due to null-recurrence).

As another example, Brownian motion time changed by the local time of a more complicated process was obtained as the limit of a diffusion in a cellular flow (\cite{HKPGY14}).
\end{Remark}

\begin{Theorem}\label{main_result2}
The law of the process $\zeta^{\e}(t)=\varepsilon^{-1/2}(Y^{\varepsilon}(t)-y(t))$ converges in distribution in $C([0,\infty),\mathbb{R}^d)$ to the solution of the stochastic differential equation
\begin{align}\label{eq:the_solution}
  d\zeta^0(t)=D_xb_1(y(t))\zeta^0(t)dt+\sqrt{\left(\int_{-\infty}^\infty (\sigma\sigma^T)(x,y(t))dx\right)} dV(t),\qquad
  \zeta^0(0)=0,
\end{align}
where $V(t)$ is as above, $\sqrt{\cdot}$ denotes the matrix square root and $D_x b_1(x)$ is the derivative tensor of the vector field $b_1$ at $x\in\mathbb{R}^d$, i.e $(D_xb_1(x))_{ij}=\partial(b_1)^{i}/\partial x^j$. The space $C([0,\infty),\reals^d)$ is endowed with the topology of uniform convergence on bounded sets $[0,T]$ for any $T>0$.
\end{Theorem}

As we will see in Section \ref{sec:aux_lemma}, integration with respect to the process $V(t)$ is well defined and the formula (\ref{eq:the_solution}) defines a well posed integral equation.

\begin{Remark} \label{rem:variation-of-parameters}
Note that our assumptions imply that equation (\ref{eq:the_solution}) has a unique solution which can be explicitly obtained by the variation of parameters formula
\[
\zeta^0(t)=\int_0^te^{\int_s^t D_x b_1(y(r))dr}\sqrt{\left( \int_{-\infty}^\infty (\sigma\sigma^T)(x,y(s)) dx \right)} dV(s).
\]
Also observe that $\zeta^0$ is not a Markovian process and in order to obtain one, the pair $(W_1,\zeta^0)$ has to be considered.
\end{Remark}

By a time-change argument, the result can be extended to some cases of non-constant $\varphi$. Let $\psi_1(y)$ and $\psi_2(x,y)$ be Lipschitz continus functions and suppose that
\begin{align}\label{eq:psi_cond}
&\sup_{y\in\mathbb{R}^d}\psi_2(.,y)\in L^1(\mathbb{R})  \cap L^2(\mathbb{R}),
&0<c_1\leq\psi_1(y) { + \psi_2(x,y)}\leq c_2<\infty
\end{align}
{ Notice that a consequence of the above assumptions is that
\[c_1 \leq \psi_1(y) \leq \psi_2(y)\]}
 Consider the system
\begin{align}
\label{eq:X_eq2} dX^{\varepsilon}(t)&=\frac{1}{\varepsilon}[\psi_1(Y^{\e}(t))+\psi_2(X^{\e}(t),Y^{\e}(t))]dW_1(t),\\
\label{eq:Y_eq2} dY^{\varepsilon}(t)&=[b_1(Y^{\varepsilon}(t))+b_2(X^{\varepsilon}(t),Y^{\varepsilon}(t))]dt+\sigma(X^{\varepsilon}(t),Y^{\varepsilon}(t))dW_2(t),
\end{align}
where $W_1$ and $W_2$ are again independent Brownian motions, $b_1, b_2$ are as in \eqref{eq:Y_eq} but assume that $b_1$ is bounded, and $\sigma$ satisfies
\begin{equation}\label{eq:new_sigma_cond}
\sup_{y\in\mathbb{R}}\frac{\Tr\sigma\sigma^T(.,y)}{(\psi_1(.)+\psi_2(.,y))^2}\in L^1(\mathbb{R})
\end{equation}
Let $\zeta^\e(t) = \e^{-1/2}(Y^\e(t) - y(t))$.

\begin{Corollary} \label{cor:non_unit_varphi}
There are independent Brownian motions (denoted again by $\tilde{W}_1$ and $\tilde{W}_2$) such that the process $(\varepsilon X^{\e}(t),\zeta^{\e}(t))$ converges weakly in $C([0,\infty),\mathbb{R}^{1+d})$ to $(X^0(t),\zeta^0(t))$ where
\begin{align*}
dX^0(t)&=\psi_1(y(t))d\tilde{W}_1,\\
d\zeta^0(t)&= D_y b_1(y(t))dt+\int_0^t\sqrt{\int_{-\infty}^{\infty}\frac{\sigma\sigma^T}{(\psi_1+\psi_2)^2}(x,y(s))dx}dV^{X^0}(s),
\end{align*}
where $V^{X^0}=\tilde{W_2}(L^{X^0}(t,0))$ and $L^{X^0}(t,0)$ is the local time of the $X^0$ process at zero.
\end{Corollary}

Theorem \ref{main_result2} will be proved in Sections \ref{sec:aux_lemma}-\ref{sec:proof_of_main_result} and the proof of Corollary \ref{cor:non_unit_varphi} will be presented in Section \ref{sec:non_unit_varphi}.

 \begin{Remark}
As it is apparent from Theorem \ref{main_result2}, the drift part of the perturbation does not contribute to the deviations of order $\sqrt{\varepsilon}$. It is natural however to conjecture that it will play a role in fluctuations of order $\varepsilon$ under some additional assumptions on the regularity of $b_1$. For example, if $\sigma\equiv 0$, almost identical arguments as in the proof of Theorem \ref{main_result2} show that $\tilde{\zeta}^{\e}(t)=\e^{-1}(Y^{\e}(t)-y(t))$ converges weakly to $\tilde{\zeta}^0(t)$ where
\[
d\tilde{\zeta}^0(t)=D_xb_1(y(t))\tilde{\zeta}^0(t)dt+\int_{-\infty}^{\infty}b_2(x,y(t))dxL^{\bar{W}_1}(dt,0).
\]
\noindent We expect that $b_2$ will also play a role on the behavior of $Y^{\varepsilon}(t)$ on timescales of order $1/\varepsilon^2$ as it takes this much time for a typical realization of the Brownian local time to make a growth of order 1. We are planning to return to these questions in a future paper.
\end{Remark}

\subsubsection*{An example: perturbed harmonic oscillator}

Let $\sigma\in L^2(\mathbb{R})$ be Lipschitz continuous and consider the equation
\[
\ddot{q}^{\e}(t)+q^{\e}(t)=\sigma(\e^{-1}W_1(t))\dot{W}_2(t),\qquad q^{\e}(0)=0,~ \dot{q}^{\e}(0)=1.
\]
Introducing $Y^{\e}=(\dot{q}^{\e},q^{\e})$ leads to the system
\begin{equation}\label{eq:perturb_harm_osc}
dY^{\e}(t)=\left(\begin{array}{cc}
0&-1\\
1&0
\end{array}\right)Y^{\e}(t)+\left(\begin{array}{c}
\sigma(\e^{-1}W_1(t))\\
0
\end{array}\right)dW_2(t),\qquad Y^{\e}(0)=(1,0).
\end{equation}
Thus $y(t)=(\cos(t),\sin(t))$.

Theorem \ref{main_result2} and Remark \ref{rem:variation-of-parameters} imply that $\zeta^{\e}=\e^{-1/2}(Y^{\e}-y)$ converges weakly to
\[
\zeta^0(t)=||\sigma||_{L^2(\mathbb{R})}\int_0^t\left(\begin{array}{c}
\cos(t-s)\\
\sin(t-s)
\end{array}\right)dV(s),
\]
where $V(t)=\bar{W}_2(L^{\bar{W}_1}(t,0))$ which in turn means
\begin{equation}\label{eq:approx}
q^{\e}(t)\approx \cos(t) +\sqrt{\e}||\sigma||_{L^2(\mathbb{R})}\int_0^t\cos(t-s)dV(s),
\end{equation}
for small enough $\e$. A typical trajectory can be seen on Figure 1.
\begin{figure}
\centering
\includegraphics[scale=0.8]{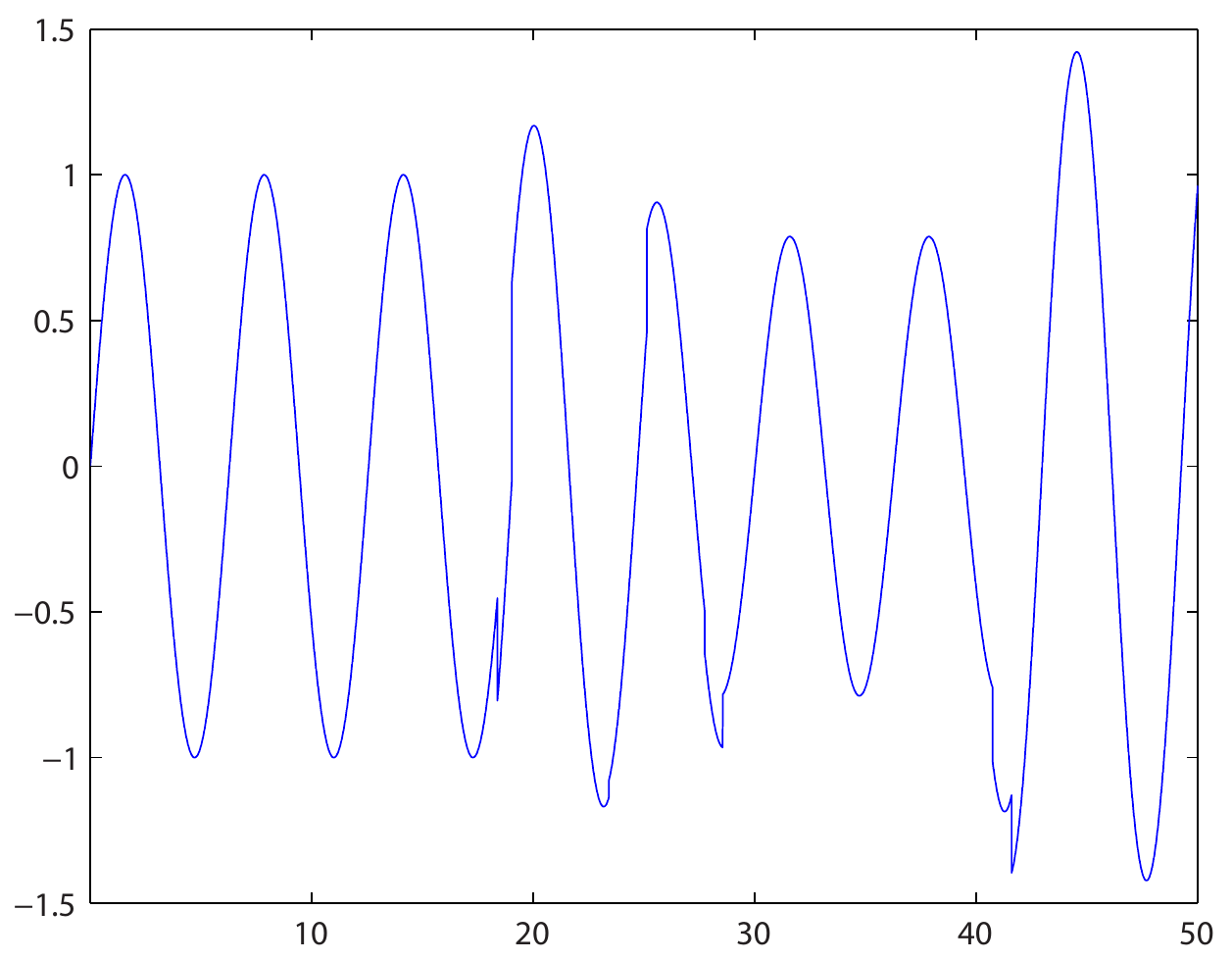}
\caption{A typical trajectory of the approximation \eqref{eq:approx} with $\sqrt{\e}=0.1$, $||\sigma||_{L^2(\mathbb{R}^2)}=100$.}
\end{figure}

\section{Auxilliary lemmas}\label{sec:aux_lemma}
In Sections \ref{sec:aux_lemma}-\ref{sec:proof_of_main_result} we prove Theorem \ref{main_result2} and therefore we consider equation \eqref{eq:the_solution}.
First we prove the  following lemma which  establishes a continuity property of certain functionals of the fast Brownian motion.

\begin{Lemma}\label{lem:mylemma}
  Suppose that $\psi \in L^1(\reals)$. Then for any $p \geq 1$, there exists $C_p>0$ such that
  \begin{equation} \label{continuity-for-integral-eq}
    \E \left| \frac{1}{\e} \int_s^t \psi(\e^{-1}W_1(r)) dr \right|^p \leq C_p |\psi|_{L^1(\reals)}^p|t-s|^{\frac{p}{2}},
  \end{equation}
\end{Lemma}

\begin{proof}
 The function
  \[\Psi(x) = \int_{-\infty}^x \psi(y) dy\]
  is well defined, continuous and bounded ($ |\Psi(x)| \leq |\psi|_{L^1(\mathbb{R})})$. Consequently, we can define
  \[f(x) = \int_0^x \Psi(y) dy\]
  so that $f''(x) = \psi(x)$.
  This function is Lipschitz continuous because
  \begin{equation} \label{f-Lipschitz-eq}
   |f(x_2)-f(x_1)|\leq \left|\int_{x_1}^{x_2} \Psi(y) dy \right| \leq |\psi|_{L^1(\mathbb{R})} |x_2-x_1|.
  \end{equation}
Note that since $f''$ is not necessarily continuous, we cannot directly apply the It$\hat{o}$-formula to $f$. By the Meyer-Tanaka formula, however, we do have
  \begin{equation} \label{Ito-eq}
    f(\e^{-1}W_1(t)) - f(\e^{-1}W_1(s)) = \frac{1}{\e} \int_s^t \Psi(\e^{-1}W_1(r)) dW_1(s) + \frac{1}{2\e^2} \int_s^t \psi(\e^{-1}W_1(r)) dr.
  \end{equation}
This implies that
  \[  \left|\frac{1}{\e} \int_s^t \psi(\e^{-1}W_1(r)) dr \right| \leq 2\e\left|f(\e^{-1}W_1(t)) - f(\e^{-1}W_1(s))\right|  + 2 \left|\int_s^t \Psi(\e^{-1}W_1(r)) dW_1(r) \right|,\]
  so that
  \[\E \left|\frac{1}{\e} \int_s^t \psi(\e^{-1}W_1(r)) dr \right|^p \leq 2^{2p-1} \left( \E \e^p |f(\e^{-1}W_1(t))-f(\e^{-1}W_1(s))|^p + \E \left|\int_s^t \Psi(\e^{-1}W_1(r)) dW_1(r) \right|^p \right).\]
  By \eqref{f-Lipschitz-eq},
  \[\e| f(\e^{-1}W_1(t)) -f(\e^{-1}W_1(s))| \leq  |\psi|_{L^1(\reals)} |W_1(t) - W_1(s)|,\]
  and we see that
  \[\E \left| \frac{1}{\e} \int_s^t \psi(\e^{-1}W_1(r)) dr \right|^p \leq C_p  \left(|\psi|_{L^1(\reals)}^p \left|t-s\right|^{\frac{p}{2}} + \left| \E \int_s^t \Psi^2(\e^{-1}W_1(r))  dr \right|^{\frac{p}{2}} \right), \]
where we used the BDG inequality.
  The integrand is bounded as $\Psi(x) \leq |\psi|_{L^1(\reals)}$ and therefore we can conclude that
  \[\E \left| \frac{1}{\e} \int_s^t \psi(\e^{-1}W_1(r)) dr \right|^p \leq C_p|\psi|_{L^1(\reals)}^p|t-s|^{\frac{p}{2}}.\]
\end{proof}

The second technical result we are going to need is an estimate on the $L^p$ convergence rate of $Y^{\varepsilon}(t)$ to $y(t)$ which is a consequence of Lemma \ref{lem:mylemma}.

\begin{Lemma}\label{lem:Y_close_to_y}
For every $p\geq 1$, there exists a constant $C_{T,p}$ such that
\[
\sup_{t\in [0,T]}\mathrm{E}|Y^{\varepsilon}(t)-y(t)|^p<C_{T,p}\varepsilon^{p/2}.
\]
\end{Lemma}

\begin{proof}
Note that
\begin{align*}
Y^{\varepsilon}(t)-y(t)&=\int_0^t(b_1(Y^{\varepsilon}(s))-b_1(y(s))ds+\int_0^tb_2\left({\varepsilon}^{-1}W_1(s),Y^{\varepsilon}(s)\right)ds+\\
&\int_0^t\sigma\left({\varepsilon}^{-1}W_1(s),Y^{\varepsilon}(s)\right)dW_2(s)=I_1^{\varepsilon}(t)+I_2^{\varepsilon}(t)+I_3^{\varepsilon}(t).
\end{align*}
By the Lipschitz continuity of $b_1$ and Jensen's inequality,
\[
\mathrm{E}|I_1^{\varepsilon}(t)|^p\leq T^{p-1}\textnormal{Lip}(b_1)^p\int_0^t\mathrm{E}|Y^{\varepsilon}(s)-y(s)|^p.
\]
On the other hand,
\[
\sup_{t\in[0,T]}\mathrm{E}|I^{\varepsilon}_2(t)|^p\leq\mathrm{E}\left(\int_0^T\hat{b}(\varepsilon^{-1}W_1(s))ds\right)^p\leq C_p \varepsilon^p T^{p/2},
\]
where in the last inequality we used Lemma \ref{lem:mylemma} with $s=0$. Finally, it is easy to see that the scalar quadratic variation of $I_3$ is
\[
\Tr<I_3>_t=\int_0^t\Tr\sigma\sigma^T(\varepsilon^{-1}W_1(s),Y^{\e}(s))ds\leq\int_0^t\hat{\sigma}^2(\varepsilon^{-1}W_1(s))ds,
\]
and therefore by the Burkholder-Davis-Gundy inequality and Lemma \ref{lem:mylemma} we have
\begin{align}\label{eq:I3}
\sup_{t\in[0,T]}\mathrm{E}|I_3^{\varepsilon}(t)|^p&\leq  C_p\mathrm{E}\left(\int_0^T\hat{\sigma}^2(\varepsilon^{-1}W_1(s))ds\right)^{p/2}<C_{p}\varepsilon^{p/2}T^{p/4}.
\end{align}
The result now follows from Gronwall's lemma.
\end{proof}

Next, we show that the stochastic integral in \eqref{eq:the_solution} is well defined.

\begin{Lemma}
  Suppose that $W(t)$ is a $d$-dimensional Wiener process and that $F(t)$ is an increasing, deterministic, real-valued function for $t \in [0,+\infty]$. Then the composition $V(t) = W(F(t))$ is a Gaussian martingale, and if $\psi(s)$ is a deterministic matrix-valued process with
  \[\int_0^t \Tr(\psi\psi^T)(s) dF(s) < +\infty, \]
  where this is the Riemann-Stieltjes integral with respect to $F$, then the stochastic integral
  \[\int_0^t \psi(s) dV(s)\]
  is a well-defined variable with distribution
  \[N\left(0, \int_0^t (\psi\psi^T)(s) dF(s)\right),\]
  where the above integral is a Riemann-Stieltjes integral.
\end{Lemma}

\begin{proof}
  If $\psi = \sum_{k=0}^{N-1} \psi_k \chi_{[t_k,t_{k+1}]}(s)$ is a step function, then
  \[\int_0^t \psi(s) d(W(F(s))) = \sum_{k=0}^{N-1} \psi_k (W(F(t_{k+1})) - W(F(t_k))), \]
  which is a zero-mean Gaussian random variable with covariance
  \[\sum_{k=0}^{N-1} \psi_k \psi_k^T (F(t_{k+1}) - F(t_k)).\]
  The result follows by the density of these step functions.
\end{proof}

\begin{Corollary} \label{cor:stoch-int-exist}
  Suppose that $W$ is a $d$-dimensional Wiener process and that $F:[0,\infty)\to\mathbb{R}$ is an increasing function that is independent of $W$. Set $V(t) = W(F(t))$. Then if $\psi(s)$ is a deterministic matrix-valued process with
  \[\E \int_0^t \Tr(\psi\psi^T)(s) dF(s) < +\infty, \]
  where this is the Riemann-Stieltjes integral with respect to $F$, then the stochastic integral
  \[\int_0^t \psi(s) dV(s)\]
  is a well-defined random variable with characteristic function for any $\lambda \in \reals^d$
  \[\E \left(\exp \left(i \left<\lambda, \int_0^t \psi(s) d(W(F(s))) \right> \right) \right)
  = \E \left(-\frac{1}{2} \int_0^t \left<(\psi\psi^T)(s)\lambda, \lambda \right> dF(s) \right). \]
  \end{Corollary}
  \begin{proof}
    This follows from the previous lemma by conditioning on $F$.
  \end{proof}

\section{A simpler case}\label{sec:reducing_to_simple}

We first show that the Theorem \ref{main_result2} holds in the case where $Y^\e(t)$ is replaced by $y(t)$ in the second argument of $\sigma$, and then we prove that the general case can be reduced to this one.
\begin{Lemma}\label{lem:easiest}
Let
\[
dJ^\e(t) = \sigma(\varepsilon^{-1}W_1(t),y(t))dW_2(t),\qquad  J^{\varepsilon}(0)=0.
\]
Then $\varepsilon^{-1/2}J^{\varepsilon}(t)$ converges in distribution to
\begin{equation} \label{eq:J-bar}
\bar{J}(t) := \int_0^t \sqrt{\left(\int_{-\infty}^\infty  (\sigma\sigma^T)(x,y(s))dx\right)} dV(s)
\end{equation}
where $V(t) = W(L^{W_1}(t,0))$ for some $d$-dimensional Wiener process, $W$, independent of $W_1$.
\end{Lemma}

\begin{proof}
Note that because $W_1$ and $W_2$ are independent, the conditional law of $\e^{-1/2}J^\e$, conditioned on $W_1$, is zero mean Gaussian and therefore it is determined by its quadratic variation. In fact, for any fixed $t>0$,
\[
\left(\varepsilon^{-1/2}J^{\varepsilon}(t) |W_1 \right)\stackrel{D}{=}N\left(0, \frac{1}{\varepsilon}\int_0^{t}(\sigma\sigma^T)(\varepsilon^{-1}W_1(s),y(s))ds\right)
\]
where the above notation means that for any $\lambda \in \mathbb{R}^d$
\begin{equation} \label{eq:ptwise-char-funct}
\E \left(\exp\left(i \left<\lambda,(\e^{-1/2} J^\e(t))\right>\right) | W_1 \right) = \exp\left(-\frac{1}{2 \e} \int_0^{t} \left<(\sigma\sigma^T)(\e^{-1} W_1(s), y(s))\lambda, \lambda \right> ds \right).
\end{equation}
We can see from the above formula that the convergence of the quadratic variation of $\e^{-1/2}J^\e(t)$ implies the converge in distribution of $\e^{-1/2}J^\e(t)$.
By the continuity of $y$, the approximating sum
\[
y_n =\sum_{i=0}^ny\left(\frac{iT}{n}\right)\mathbbm{1}_{[iT/n,(i+1)T/n)}
\]
converges to $y$ in $C([0,T]; \reals^d)$. Because of the continuity of $\sigma$, for any fixed $\e$,
 \[\frac{1}{\e} \int_0^t (\sigma\sigma^T)(\e^{-1} W_1(s), y_n(s)) \to \frac{1}{\e} \int_0^t (\sigma\sigma^T)(\e^{-1}W_1(s), y(s)) ds\]
almost surely.

 By the definition of the Brownian local time, it is not hard to check that on a set of full measure we have
\begin{align*}
\frac{1}{\varepsilon}\int_0^t(\sigma\sigma^T)&(\varepsilon^{-1}W_1(s),y_n(s))ds=\\
&=\frac{1}{\varepsilon}\sum_{i=0}^n\int_{\mathbb{R}}(\sigma\sigma^T)\left(\frac{x} {\varepsilon},y\left(\frac{it}{n}\right)\right)\left(L^{W_1} \left({\frac{(i+1)t}{n}},x \right)-L^{W_1}\left(\frac{it}{n},x\right)\right)dx.
\end{align*}
By a change of variables, this becomes
\[
\sum_{i=0}^n\int_{\mathbb{R}}(\sigma\sigma^T)\left(x,y\left(\frac{it}{n}\right)\right)\left(L^{W_1}\left(\frac{(i+1)t}{n}, \varepsilon x\right)-L^{W_1}\left(\frac{it}{n},\varepsilon x\right)\right)dx.
\]
Taking $n\to\infty$, we see that for fixed $\e$, we have with probability one that
\begin{equation}\label{eq:finite_epsilon}
\frac{1}{\e} \int_0^t (\sigma\sigma^T)(\e^{-1} W_1(s),y(s)) ds = \int_\reals \int_0^t (\sigma\sigma^T)(x, y(s)) L^{W_1}(ds,\e x)dx,
\end{equation}
where the right hand side is the Riemann-Stieltjes integral with respect to the increasing function $t \mapsto L^{W_1}(t,\e x)$.
Lastly, we argue that this converges as $\e \to 0$ to
\[\int_0^t \left( \int_\reals (\sigma\sigma^T)(x,y(s)) dx \right) L^{W_1}(ds,0). \]

Similarly to (\ref{eq:finite_epsilon}),
\[
\frac{1}{\e} \int_0^t (\sigma\sigma^T)(\e^{-1} W_1(s),y(s)) \mathbbm{1}_{\{\e^{-1}|W_1(s)|\leq N\}} ds = \int_{-N}^N \int_0^t (\sigma\sigma^T)(x, y(s)) L^{W_1}(ds,\e x) ds,
\]
and consequently for any $N \in \nat$,
\[\begin{array}{l}
\ds{\left| \frac{1}{\e} \int_0^t (\sigma\sigma^T)(\e^{-1} W_1(s),y(s)) ds - \int_0^t \left( \int_\reals (\sigma\sigma^T)(x,y(s)) dx \right) L^{W_1}(ds,0) \right|}\\
\ds{\leq \left| \frac{1}{\e} \int_0^t (\sigma\sigma^T)(\e^{-1} W_1(s),y(s)) ds - \frac{1}{\e} \int_0^t (\sigma\sigma^T)(\e^{-1}W_1(s),y(s)) \mathbbm{1}_{\{\e^{-1}|W_1(s)|\leq N\}} ds \right|}\\
\ds{+ \left| \int_{-N}^N \int_0^t (\sigma\sigma^T)(x, y(s)) L^{W_1}(ds,\e x) dx - \int_{-N}^N \int_0^t (\sigma\sigma^T)(x,y(s)) L^{W_1}(ds,0) dx\right|}\\
\ds{+\left|\int_{-N}^N \int_0^t (\sigma\sigma^T)(x,y(s)) L^{W_1}(ds,0) dx - \int_\reals \int_0^t (\sigma\sigma^T)(x,y(s)) L^{W_1}(ds,0) dx\right|}\\
\ds{:=I_1 + I_2 + I_3}.
\end{array}\]
where $|.|=|.|_{L(\reals^d,\reals^d)}$ is the operator norm.

By Lemma \ref{lem:mylemma},
\[\begin{array}{l}
\ds{\sup_{\e \geq 0} \E I_1
= \sup_{\e \geq 0} \E \frac{1}{\e} \left| \int_0^t (\sigma\sigma^T)(\e^{-1}W_1(s),y(s)) \mathbbm{1}_{\{\e^{-1}|W_1(s)|> N\}}ds \right|}\\
\ds{\leq \sup_{\e \geq 0} \E  \frac{1}{\e} \int_0^t \hat{\sigma}^2(\e^{-1}W_1(s)) \mathbbm{1}_{\{\e^{-1}|W_1(s)| > N\}} ds   }
\ds{\leq |\hat{\sigma}^2\chi_{\{|x|>N\}}|_{L^1(\reals)} \sqrt{t}.}
\end{array}\]
where we used that for the positive matrix $\sigma\sigma^T$,
\[
|\sigma\sigma^T|=\lambda_{max}\leq\sum_{i=1}^d\lambda_i=\Tr\sigma\sigma^T.
\]
Because $\hat{\sigma}^2 \in L^1(\reals)$, we can choose $N$ large enough to make this contribution arbitrarily small. Similarly, with probability $1$,
\[\lim_{N \to 0} \int_{-N}^N \int_0^t (\sigma\sigma^T)(x,y(s)) L^{W_1}(ds,0) dx = \int_\reals \int_0^t (\sigma\sigma^T)(x,y(s)) L^{W_1}(ds,0) dx \]
so that $I_3 \to 0$ almost surely.
Next, we claim that for any $N \in \nat$,
\[\lim_{\e \to 0} \int_{-N}^N \int_0^t (\sigma\sigma^T)(x, y(s)) L^{W_1}(ds,\e x) dx = \int_{-N}^N \int_0^t (\sigma\sigma^T)(x,y(s)) L^{W_1}(ds,0) dx.\]
Indeed,
\[\lim_{\e \to 0} \sup_{|x|\leq N} \sup_{0 \leq s \leq t} |L^{W_1}(s,\e x) - L^{W_1}(s,0)| = 0, \]
which implies the weak convergence of the Lebesgue-Stieltjes measures $L^{W_1}(ds,\e x)$ to $L^{W_1}(ds,0)$ uniformly for $|x|<N$. Therefore $I_2 \to 0$ with probability $1$.

Because $I_1 \to 0$ in $L^1(\Omega)$, for any subsequence $\e_n \to 0$, there exists a further subsequence $\e_{n_k}$ such that
\[\lim_{k \to +\infty} \frac{1}{\e_{n_k}} \int_0^t  (\sigma\sigma^T)(\e_{n_k}^{-1}W_1(s), y(s)) ds = \int_\reals \int_0^t (\sigma\sigma^T)(x,y(s)) L^{W_1}(ds,0) dx\]
almost surely.

Then, from \eqref{eq:ptwise-char-funct} and the dominated convergence theorem, we see that
\begin{align*}
\lim_{k\to+\infty} \E &\left(\exp\left(i \left<\lambda,(\e_{n_k}^{-1/2} J^{\e_{n_k}}(t))\right>\right) \right)
=\\
&=\E \left(\exp \left(-\frac{1}{2} \int_0^t \int_{-\infty}^\infty \left< (\sigma\sigma^T)(x,y(s))\lambda, \lambda \right>dx L^{W_1}(ds,0) \right) \right).
\end{align*}
Since this is true for a subsequence of every sequence $\e_n \to 0$, we conclude that
\begin{align*}
\lim_{\e \to 0} \E &\left(\exp\left(i \left<\lambda,(\e^{-1/2} J^{\e}(t))\right>\right) \right)
=\\
&=\E \left(\exp \left(-\frac{1}{2} \int_0^t \int_{-\infty}^\infty \left< (\sigma\sigma^T)(x,y(s))\lambda, \lambda \right>dx L^{W_1}(ds,0) \right) \right).
\end{align*}

From Corollary \ref{cor:stoch-int-exist}, this is equal to the characteristic function of $\bar{J}(t)$ given by \eqref{eq:J-bar}.
We have proven that for fixed $t$,
\[\e^{-1/2}(J^\e(t)) \to (\bar{J}(t)) \text{ in distribution}.\]
By the same arguments, we can show that for any finite collection of times $0\leq t_1 < t_2 < ... < t_n$, the finite dimensional distributions
\[\e^{-1/2}\vec{J}^\e:= (\e^{-1/2}J^\e(t_1), ..., \e^{-1/2}J^\e(t_n)) \to (\bar{J}(t_1), ... \bar{J}(t_2)) \text{ in distribution}.\]

It remains to show that the laws of $\{\e^{-1/2} J^\e\}$ are tight as a family of measures on $C([0,T];\reals^n)$. Exactly as in (\ref{eq:I3}),  we have
\[
\ds{\sup_{\e > 0} \E \left|\e^{-1/2} J^\e(t) - \e^{-1/2} J^\e(s) \right|_{\reals^d}^p \leq  C_{p}}|t-s|^{\frac{p}{4}}\qquad p\geq 1
\]
and therefore this family of measures is tight by the Kolmogorov test. Since the family is tight and all of the finite dimensional distributions converge, our conclusion follows.

\end{proof}

Now introduce the process $Z^{\varepsilon}$ that uniquely solves the integral equation
\begin{equation}\label{eq:Z_equation}
Z^{\varepsilon}(t)=y_0+\int_0^tb_1(Z^{\varepsilon}(s))ds+J^{\varepsilon}(t).
\end{equation}
The following lemma shows that the general case can be reduced to studying \eqref{eq:Z_equation}.

\begin{Lemma}\label{lem:reducing_to_simple}
We have
\[
\varepsilon^{-1}\E\sup_{t\in[0,T]}|Y^{\varepsilon}(t)-Z^{\varepsilon}(t)|^2\to 0\qquad \varepsilon\to 0.
\]
\end{Lemma}

\begin{proof}
Note that if $K^{\varepsilon}(t)=\varepsilon^{-1/2}(Y^{\varepsilon}(t)-Z^{\varepsilon}(t))$ then
\begin{align*}
K^{\varepsilon}(t)&=\varepsilon^{-1/2}\int_0^t\left(b_1(Y^{\varepsilon}(s))-b_1(Z^{\varepsilon}(s))\right)ds+\varepsilon^{-1/2}\int_0^tb_2(\varepsilon^{-1}W_1(s),Y^{\varepsilon}(s))ds+\\
&+\varepsilon^{-1/2}\int_0^t(\sigma(\varepsilon^{-1}W_1(t),Y^{\varepsilon}(s))-\sigma(\varepsilon^{-1}W_1(t),y(s)))dW_2(t)=I_1^{\varepsilon}(t)+I^{\varepsilon}_2(t)+I^{\varepsilon}_3(t).
\end{align*}
For the second term, we write
\[
\mathrm{E}\sup_{t\in[0,T]}|I^{\varepsilon}_2(t)|^2\leq\mathrm{E}\frac{1}{\varepsilon}\left(\int_0^T\hat{b}(\varepsilon^{-1}W_1(s))ds\right)^2\leq \varepsilon C_{p} T,
\]
where in the last inequality we used Lemma \ref{lem:mylemma} with $p=2$.

On the other hand, by Doob's inequality, 
\begin{align*}
\mathrm{E}\sup_{t\in[0,T]}|I_3^{\varepsilon}(t)|^2\leq \frac{4}{\varepsilon}\mathrm{E}\int_0^T\Tr\tilde{\sigma}\tilde{\sigma}^T(t)dt=M^{\varepsilon}_>+M^{\varepsilon}_<,
\end{align*}
where $M^{\varepsilon}_>$ (resp. $M^{\varepsilon}_<$) is the part of the sum-integral when $|\varepsilon^{-1}W_1(t)|>N$ (resp. $|\varepsilon^{-1}W_1(t)|\leq N$) and
\[
\tilde{\sigma}(t)=\sigma(\e^{-1}W_1(t),Y^{\e}(t))-\sigma(\e^{-1}W_1(t),y(t)).
\]

We have by a simple calculation
\[
M^{\varepsilon}_>\leq \frac{16}{\varepsilon}\mathrm{E}\int_0^T\mathbbm{1}_{\{W_1(t)>\varepsilon N\}}\hat{\sigma}^2(\varepsilon^{-1}|W_1(t)|)dt\leq C_{T}|\hat{\sigma}\mathbbm{1}_{\{|x|>N\}}|_{L^2(\mathbb{R})}^2
\]
where we used Lemma \ref{lem:mylemma} with $p=2$ and $\psi(x)=\hat{\sigma}^2(x)\mathbbm{1}_{\{|x|>N\}}$. Pick $\delta>0$ and let $N$ be so large such that this is less than $\delta/2$.

For the other term, we use the Lipschitz continuity of $\sigma$ and the Cauchy-Schwartz inequality to write
\[
M^{\varepsilon}_<\leq C_{Lip(\sigma)}\left(\frac{1}{\varepsilon}\int_0^T\mathrm{E}|Y^{\varepsilon}(s)-y(s)|^4ds\right)^{1/2}\left(\mathrm{E}\frac{1}{\varepsilon}\int_0^T\mathbbm{1}_{\{\varepsilon^{-1}|W_1(s)|\leq N\}}ds\right)^{1/2}.
\]
The second term in the product is bounded by $CN^{1/2}T^{1/4}$, while the first one is $\mathcal{O}(\sqrt{\varepsilon})$. Indeed, Lemma \ref{lem:Y_close_to_y} implies $E|Y^{\varepsilon}(t)-y(t)|^4<C_{T}\varepsilon^2$. Now choose $\epsilon$ small enough such that $M^{\varepsilon}_<<\delta/2$. The result now follows from the Lipschitz continuity of $b_1$ and Gronwall's lemma since $\delta$ is arbitrary.
\end{proof}

\section{Proof of Theorem \ref{main_result2}}\label{sec:proof_of_main_result}

We are going to prove that the processes $\varepsilon^{-1/2}(Z^{\varepsilon}(t)-y(t))$ converge weakly to the limit in (\ref{eq:the_solution}). It is not hard to see that as a consequence of the tightness of $\varepsilon^{-1/2}J^{\varepsilon}$, the Lipschitz continuity of $b_1$, Lemma \ref{lem:reducing_to_simple} and Lemma \ref{lem:Y_close_to_y}, we also have

\begin{Lemma}\label{lem:tightness}
The family of processes  $\varepsilon^{-1/2}(Z^{\varepsilon}(t)-y(t))$ is tight in $C([0,T])$.
\end{Lemma}

Now we prove Theorem \ref{main_result2}.
\begin{proof}[Proof of Theorem \ref{main_result2}]
Let $\e_n$ be any sequence converging to $0$. By Lemma \ref{lem:tightness}, there is a subsequence $\e_{n_k}$ such that the laws of $\{\e_{n_k}^{-1/2} (Z^{\e_{n_k}}(t) - y(t))\}$ converge weakly. By the Skorokhod representation theorem and Lemma \ref{lem:easiest}, we can assume that $\varepsilon_{n_k}^{-1/2}(Z^{\varepsilon_{n_k}}-y)\to\bar{Z}$ and $\varepsilon_{n_k}^{-1/2}J^{\varepsilon_{n_k}}\to \bar{J}$ in $C[0,T]$ almost surely on some probability space.

By Taylor approximation, we have
\[
\frac{Z^{\varepsilon_{n_k}}(t)-y(t)}{\sqrt{\varepsilon_{n_k}}}=\int_0^tD_xb_1(y(s))\frac{Z^{\varepsilon_{n_k}}(s)-y(s)}{\sqrt{\varepsilon_{n_k}}}ds+\frac{1}{\sqrt{\varepsilon_{n_k}}}J^{\varepsilon_{n_k}}(t)+R(t,\varepsilon_{n_k}),
\]
where
\[
\mathrm{E}\sup_{t\in[0,T]}|R(t,\varepsilon_{n_k})|\leq C \int_0^T\varepsilon_{n_k}^{-1/2}\mathrm{E}|Z^{\varepsilon_{n_k}}(u)-y(u)|^2du.
\]
The constant in the above formula depends on the second derivatives of $b_1$ which are bounded by assumption. Taking $k\to\infty$ and using Lemma \ref{lem:reducing_to_simple} and Lemma \ref{lem:Y_close_to_y}, the above expression converges to
\[\bar{Z}(t) = \int_0^t D_xb_1(y(s)) \bar{Z}(s) ds  + \bar{J}(t).\]
By uniqueness, $\bar{Z} = \zeta^0$ in distribution where $\zeta^0$ is defined by \eqref{eq:the_solution}.
Because our original sequence $\e_n$ was arbitrary,
\[\lim_{\e \to 0} \e^{-1/2} (Z^\e - y) = \bar{Y} \text{ in distribution}.\]
Finally, by Lemma \ref{lem:reducing_to_simple}, $\e^{-1/2}(Y^\e(t) - y(t))$ and $\e^{-1/2} ( Z^\e(t) - y(t))$ have the same limit, so our result follows.
\end{proof}

\section{Proof of Corollary \ref{cor:non_unit_varphi}}\label{sec:non_unit_varphi}

 We first prove Corollary \ref{cor:non_unit_varphi} in the case where $b_1\equiv 0$. That is, in the case where the $Y^\e(t) \to 0$ uniformly on bounded time intervals. In this case, we can time change the system $(X^\e,Y^\e)$ which solves \eqref{eq:X_eq2}-\eqref{eq:Y_eq2} into a system where the fast motion is Brownian.
 
 Let $(X^{\varepsilon}(t),Y^{\varepsilon}(t))$ solve \eqref{eq:X_eq2} and \eqref{eq:Y_eq2} with $b_1\equiv0$, and introduce the time-change
\[
s^{\varepsilon}(t)=\int_0^t(\psi_1(Y^{\e}_u)+\psi_2(X^{\varepsilon}(u),Y^{\varepsilon}(u)))^2du.
\]
Denote the inverse of $s^\e(t)$ by $t^{\varepsilon}(s)$. If we set $\tilde{X}^\e(s) = X^\e(t^\e(s))$, then it becomes fast Brownian motion, in the sense that there exists $\tilde{W}_1$ such that $\tilde{X}^\e(s) = \e^{-1}\tilde{W}_1(s)$. If we set $\tilde{Y}^\e(s) = Y^\e(t^\e(s))$, then it solves the SDE
\[d \tilde{Y}^\e(s) = \tilde{b}(\e^{-1}\tilde{W}_1(s), \tilde{Y}^\e(s)) ds + \tilde{\sigma}(\e^{-1} \tilde{W}_1(s), \tilde{Y}^\e(s)) dW_2(s),\]
where
\[
\tilde{b}(x,y)=\frac{b_2(x,y)}{(\psi_1(y)+\psi_2(x,y))^2}
\]
and
\[
\tilde{\sigma}(x,y)=\frac{\sigma(x,y)}{(\psi_1(y)+\psi_2(x,y))}.
\]
By \eqref{eq:psi_cond} and \eqref{eq:new_sigma_cond}, we have $\sup_{y\in\mathbb{R}^d}\tilde{b}_2(.,y)$, and $ \sup_{y\in\mathbb{R}^d}\Tr\tilde{\sigma}\tilde{\sigma}^T(.,y)\in L^1(\mathbb{R})$.

By Theorem \ref{main_result2}, we have that if 
 $\tilde{\zeta}^{\e}(s)=\frac{1}{\sqrt{\e}}(\tilde{Y}^{\e}(s)$, then we have
\begin{equation}\label{eq:pre_tc_conv}
(\varepsilon \tilde{X}^{\e}(s),\tilde{\zeta}^{\e}(s))\stackrel{\e\to 0}\Rightarrow(\tilde{W_1}(s),\tilde{\zeta}^0(s))\qquad \textrm{in  } C([0,\infty),\mathbb{R}^{1+d}),
\end{equation}
where $\tilde{\zeta}^0(s)$ is the limit as in Theorem \ref{main_result2} but with $\tilde{b}_1\equiv 0$, $\tilde{\sigma}$ and some $\tilde{W}_2$ independent of $\tilde{W}_1$.

Using \eqref{eq:psi_cond}, Lemma \ref{lem:mylemma}, and Lemma \ref{lem:Y_close_to_y}, it is not hard to see that
\[
t^{\e}(s)=\int_0^s\frac{1}{(\psi_1(\tilde{Y}^{\e}(v))+\psi_2(\e^{-1}\tilde{W_1}(v), \tilde{Y}^{\e}(v)))^2}dv\to \int_0^s\frac{1}{\psi^2_1(\tilde{y}(v))}dv=:t^0(s),
\]
where the convergence takes place in $L^2(\Omega)$ uniformly on bounded intervals. This also implies, by \eqref{eq:psi_cond}, that
\begin{equation}\label{eq:limit_tc}
s^{\e}(t)\to s^{0}(t)=\int_0^t\psi_1^2(y(u))du\qquad \textrm{in } L^2(\Omega)\textrm{, uniformly on bounded intervals}.
\end{equation}
Therefore, by reversing the time change, we see that
\[(\e X^\e(t), \zeta^\e(t)) = ( \tilde{W}_1^\e(s^\e(t)), \tilde{\zeta}^\e(s^\e(t))) \Rightarrow (\tilde{W}_1(s^0(t)), \tilde{\zeta}^0(s^0(t)))=: (X^0, \zeta^0).\]
Since $s^0$ is a deterministic time-change,
\begin{align*}
&\int_{-\infty}^{\infty}f(x)L^{\tilde{W}_1}(s^0(t),x)dx=\int_0^{s^0(t)}f(\tilde{W}_1(u))du=\\
&=\int_0^tf(\tilde{W}_1(s^0(v)))\psi_1^2(y(v))dv=\int_0^tf(\tilde{W}_1(s^0(v)))d\langle\tilde{W}_1(s^0(v))\rangle = \int_0^tf(\tilde{W}_1(s^0(v)))d\langle X^0\rangle,
\end{align*}
for every bounded, measurable test functions, where $\langle.\rangle$ denotes the quadratic variation. Therefore, $L^X(t,x)=L^{\tilde{W}_1}(s^0(t),x)$ is the local time of $\tilde{W}_1(s^0(v))$. This implies that $\tilde{V}(s^0(t))=\tilde{W}_2(L^X(t,0))$.

Using this, \eqref{eq:pre_tc_conv}, \eqref{eq:limit_tc}, and that $s^0$ is a deterministic time change, the proof can be concluded.

If $b_1 \not \equiv 0$, then the time change argument described above does not immediately convert $(\tilde{X}^\e, \tilde{\zeta}^\e)$ into a pair that fulfills the assumptions of Theorem \ref{main_result2}. The main reason for this is that the time change converts $y(t)$, which is the solution to a deterministic ODE, into $\tilde{y}(s):= \tilde{y}(t^\e(s))$ which solves
\[\frac{d \tilde{y}}{ds} = \frac{b_1(\tilde{y}(s))}{(\psi_1(\tilde{Y}^\e(s)) + \psi_2(\tilde{X}^\e(s), \tilde{Y}^\e(s)))^2}.\] 

Before doing the time change, we notice that $\zeta^\e(t) := \e^{-1/2}(Y^\e(t) - y(t))$ solves
\begin{align*}
  d\zeta^\e(t) = &\e^{-1/2} (b_1(y(t) + \sqrt{\e}\zeta^\e(t)) - b_1(y(t))) dt + \e^{-1/2} b_2(X^\e(t), y(t) + \sqrt{\e}\zeta^\e(t)) dt\\ 
  &+ \e^{-1/2} \sigma(X^\e(t), y(t)+\sqrt{\e}\zeta(t)) dW_2(t).
\end{align*}
We can prove Corollary \ref{cor:non_unit_varphi} by analyzing these three terms separately. By the properties of derivatives, if $\zeta^\e \to \zeta^0$, the first term
\[\int_0^T \e^{-1/2} (b_1(y(t) + \sqrt{\e}\zeta^\e(t)) - b_1(y(t))) dt \to \int_0^T D_y b_1(y(t))\cdot\zeta^0(t) dt.\]
We can then use the tightness of $\{\zeta^\e\}$ to extract a convergent subsequence.
The second term with $b_2$ converges to zero uniformly on bounded intervals of time because of Lemma \ref{lem:mylemma}. Finally, we can analyze the stochastic term using the time change argument described at the beginning of this section. The results of Corollary \ref{cor:non_unit_varphi} will follow. 

\subsubsection*{Acknowledgements}

{\small
The authors are grateful to D.\ Dolgopyat for introducing them to the problem and to L. Koralov and D. Dolgopyat for their helpful suggestions during invaluable discussions and for proofreading the manuscript. While working on the paper, Z.\ Pajor-Gyulai was partially supported by the NSF grants number 1309084 and DMS1101635. M. Salins was partially supported by the NSF grant number 1407615.
}


\begin{thebibliography}{10}

\bibitem{BAC07} G. Ben Arous, J. \v Cern\'y,
\textit{Scaling limit for trap models on $\mathbb{Z}^d$},
The Annals of Probability 35, no. 6, pp. 2356--2384 (2007).

\bibitem{FW04} M. Freidlin, A. Wentzell, \textit{Diffusion processes on an open book and the averaging principle}, Stochastic Processes and their Applications 113, pp. 101--126. (2004).

\bibitem{FW12} M. Freidlin, A. Wentzell, \textit{Random Perturbations of Dynamical Systems}, Springer-Verlag, New York (2012).

\bibitem{GS82} I. Gikhman, A. Skorokhod, \textit{Stochastic Differential Equations and their applications}, Springer-Verlag, New York (1972).

\bibitem{HKPGY14} M. Hairer, L. Koralov, Z. Pajor-Gyulai, \textit{Averaging and Homogenization in cellular flows - An exact description of the phase transition}, Submitted to Annales d'Institut Henri Poincare (2014).

\bibitem{HM11} M. Hairer, C. Manson, \textit{Periodic homogenization with an interface: The multidimensional case}, The Annals of Probability 39, no. 2, pp. 648--682 (2011).

\bibitem{K68} R. Z. Khasminskii, \textit{On averaging principle for Ito stochastic differential equations}, Kybernetika, Chekhoslovakia 4, no. 3 pp. 260--279 (in Russian) (1968).

\bibitem{KK04} R. Z. Khasminskii, N. Krylov, \textit{On averaging principle for diffusion processes with null-recurrent fast component}, Stochastic Processes and their Applications 93, pp. 229-240 (2001).

\bibitem{K04} Y. Kifer, \textit{Some recent advances in averaging}, Modern Dynamical Systems and Applications, Cambridge University Press, Cambridge UK, pp. 385--403 (2004).

\bibitem{KY04} R.Z. Khasminskii, G. Yin, \textit{On averaging principlses: an asymptotic expansion approach}, SIAM Journal on Mathematical Analysis 35, no. 6,  pp. 1534--1560 (2004).

\bibitem{P77} G.C. Papanicolaou, D. Stroock, S.R.S. Varadhan, \textit{Martingale approach to some limit theorems}, Duke Turbulence Conference (1977).

\bibitem{PS08} G.A. Pavliotis, A. Stuart, \textit{Multiscale Methods: Averaging and Homogenization}, Texts in Applied Mathematics 53, Springer, New York (2008).

\bibitem{S89} A. Skorokhod, \textit{Asymptotic Methods in the Theory of Stochastic Differential Equations}, Translations of Mathematical Monograhps, Vol. 78, AMS, Providence, RI (1989).

\end{thebibliography}
\end{document}